\documentclass[11pt]{article}
\usepackage{amssymb, latexsym, mathtools, amsthm, amsmath}
\usepackage{authblk}
\begingroup\makeatletter
\@for\theoremstyle:=definition,remark,plain\do{\expandafter\g@addto@macro\csname th@\theoremstyle\endcsname{
\addtolength\thm@preskip\parskip}}
\endgroup
\usepackage[margin=1in]{geometry}
\renewenvironment{abstract}
 { \normalsize
  \list{}{\setlength{\leftmargin}{.0cm}%
    \setlength{\rightmargin}{\leftmargin}}%
  \item {\bf \abstractname.}\relax}
 {\endlist}
\usepackage{titlesec}
 \titleformat{\paragraph}
{\normalfont\normalsize\em\bfseries}{\theparagraph}{1em}{$\blacktriangleright$\hspace{0.2cm}}
\titlespacing*{\paragraph}{0pt}{3.25ex plus 1ex minus .2ex}{0.5ex plus .2ex}
\usepackage{txfonts}
\usepackage{booktabs}
\usepackage{pgf,tikz}
\usetikzlibrary{chains}
\usetikzlibrary{decorations, decorations.pathmorphing}
\usetikzlibrary {shapes}

\definecolor{dnrbl}{rgb}{0,0,0.3}
\definecolor{dnrgr}{rgb}{0,0.3,0}
\definecolor{dnrre}{rgb}{0.5,0,0}
\usepackage[colorlinks=true, citecolor=dnrgr, linkcolor=dnrre, urlcolor=dnrre] {hyperref}
\usepackage{xcolor,colortbl}
\usepackage{enumerate}
\usepackage{pdfsync}
\usepackage[numbers]{natbib}
\usepackage{parskip}
\usepackage{tabularx}
\theoremstyle{plain}
\newtheorem{thm}{Theorem}[section]
\newtheorem{prop}[thm]{Proposition}
\newtheorem*{propa}{Proposition}

\newtheorem{lem}[thm]{Lemma}

\theoremstyle{definition}

\newtheorem{defi}[thm]{Definition}
\newcommand{\Nat}{\mathbb{N}}

\newcommand{\restr}{\upharpoonright}  %restriction

\newcommand{\brag}[1]{{\langle #1\rangle}}

\newcommand{\sqbrad}[2]{\{\hspace{0.03cm}{#1} : {#2}\hspace{0.03cm}\}}

\DeclarePairedDelimiter{\dbra}{\llbracket}{\rrbracket}

\newcommand{\CC}{\mathcal{C}}

\newcommand{\abs}[1]{|{#1}|}

\newcommand{\parb}[1]{\big({#1}\big)}
\newcommand{\parB}[1]{\Big(\hspace{0.04cm}{#1}\hspace{0.04cm}\Big)}
\newcommand{\parlr}[1]{\left({#1}\right)}
\newcommand{\ou}{oracle-use\ }

\newcommand{\TT}{\mathcal{T}}

\newcommand{\ml}{Martin-L\"{o}f }

\newcommand{\pz}{$\Pi^0_1$\ }

\newcommand{\ie}{i.e.\ }
\newcommand{\ce}{c.e.\ }

\newcommand{\pf}{prefix-free }

\newcommand{\GG}{\mathcal G}

\newcommand{\FF}{\mathcal F}

\newcommand{\twome}{2^{\omega}}
\newcommand{\zp}{\mathbf{0}'}
\newcommand{\twomel}{2^{<\omega}}

\newcommand{\wedga}{\ \wedge\ \ }

\newcommand{\PA}{\mathsf{PA}}

\newcommand{\hthree}{\hspace{0.3cm}}

\newcommand{\Hs}{\mathsf{H}}

\newcommand{\hit}{{\rm hit}}

  \usepackage{tcolorbox}
 
\title{Growth and irreducibility in path-incompressible trees\thanks{%Version: \today. 
Supported by NSFC-11971501 and partially by the  
program of the Institute for Mathematical Sciences, National University of Singapore, 2021. We thank Wang Wei for corrections in early drafts of this work.}}
\author{George Barmpalias}\author{Xiaoyan Zhang} 
\affil{State Key Lab of Computer Science, Inst.\ of Software, Chinese Acad.\ of Sciences, Beijing, China\vspace{0.1cm}}

\begin{document}
\maketitle
\begin{abstract}
We study randomness-preserving transformations of path-incompressible trees, namely trees of finite randomness deficiency. 
We characterize their  branching density, and show: 
(a) sparse perfect path-incompressible trees can be effectively densified, almost surely;
(b) there exists a path-incompressible tree with infinitely many paths which does not compute any 
 perfect path-incompressible tree with computable oracle-use.
\end{abstract}
\section{Introduction}\label{1O8WCJrzWz}
Algorithmic randomness  appears in different forms: infinite bit-sequences ({\em reals}), 
arrays, trees, and structures.
It is often essential to
effectively transform one form into another, without sacrificing algorithmic complexity:  
transforming  a real which is random with respect to a Bernoulli distribution into a uniformly random real
goes back to \citet{vonNeumann1951}, and 
general cases, including non-computable distributions,  have been explored  \citep{BienvenuCoin2012}. 

We study the hardness of transformations of closed sets of random points (or trees with incompressible paths) 
with respect to dimensionality features: branching and  accumulation points.
We use $\sigma,\tau,\rho,\eta$  for  bit-strings,  $x,y,z$  for reals, and
let $K(\sigma)$ be the \pf Kolmogorov complexity of  
$\sigma$:  the length of the shortest self-delimiting  program generating $\sigma$.
\begin{defi}
The (randomness) {\em deficiency} of $\sigma$ is $|\sigma|-K(\sigma)$. 
The {\em deficiency} of a set of strings  is the supremum of the deficiencies of its members. 
The deficiency of $x$  is the deficiency of the set of its prefixes.
\end{defi}
A real  is  {\em random} if it has finite deficiency; this is equivalent to the  notion of \citet{MR0223179}.
A set  of strings is a {\em tree} if it is downward-closed with respect to the prefix relation $\preceq$. 
A {\em path through tree $T$} is a real with all its prefixes in $T$.
A tree $T$  is 
\begin{itemize}
\item {\em perfect} if each $\sigma\in T$ has at least two $\preceq$-incomparable strict extensions in $T$ 
\item {\em pruned} if each $\sigma\in T$ has at least one proper extension in $T$
\item {\em proper} if it has infinitely many paths; {\em positive} if the measure of  its paths is positive
\item {\em effectively proper} there is an unbounded computable lower bound on $\abs{T\cap 2^n}$
\item {\em path-incompressible} if it is pruned and has finite deficiency, as a set of strings
\item  {\em path-random} if it is pruned and all of its paths are 
random\footnote{all infinite trees we consider are assumed to be pruned, hence representing  closed sets  of $\twome$.}
\end{itemize}
where $2^n$ denotes the set of $n$-bit strings.
Path-random trees are not always path-incompressible:
it is possible that $T$ is path-random and $T$ has infinite deficiency.  
On the other hand by \citep{deniscarlsch21,treeout} 
every perfect path-random tree computes a path-incompressible tree.

Our purpose  is to determine when randomness can be effectively manipulated with respect
to topological or density characteristics (length distribution, accumulation points etc.)

To this end, we ask: within the path-incompressible trees 
\begin{enumerate}[\hspace{0.3cm}(a)]
\item can every tree compute a proper  tree?
\item can every perfect  tree compute a positive  tree?
\item can every proper  tree compute a perfect  tree?
\item can every sparse perfect  tree be effectively transformed into a denser tree?
\end{enumerate}
These questions 
are  tree-analogues of problems of randomness extraction such as effectively
increasing the Hausdorff dimension of  reals \citep{dimiller}, except that the increase is now  on the structural density
of the tree, without loss in the algorithmic complexity of its paths.

\citet[Theorem 5.3, Theorem 7.7]{bslBienvenuP16} 
showed that incomplete  randoms do not compute effectively-proper path-incompressible trees.
\citet{deniscarlsch21} showed that sufficiently random reals do not compute perfect path-random trees.
A complete  answer to (a) was given in 
\citep{treeout}:  incomplete randoms do not compute perfect path-random trees.

A negative answer to question (b)
was given in \citep{treeout}: there exists a
perfect path-incompressible tree which cannot computably enumerate any positive path-random tree.
Much of this work was motivated by the study of compactness  in fragments of second-order arithmetic.

\subsection*{Outline of our results}
Our main contribution is a partial negative answer to question (c): 

\begin{thm}\label{dad73d1d}
There is a path-incompressible effectively proper tree that does not compute any path-random perfect tree with a computable upper bound on the
oracle-use.
\end{thm}

More generally,  we show that for any function $f$ there is a path-incompressible effectively proper tree $T\leq_T\emptyset'\oplus f$  which does not compute any path-random perfect tree within \ou $f$.
The difficulty of constructing a tree with these properties is due to (i) it has to be effectively proper; (ii) it has to be path-incompressible, rather than merely path-random; (iii) there is no restriction on the density of branching of the trees
we want to avoid computing. 

If some of these conditions are relaxed,  question (c) admits a simple answer:
\begin{prop}\label{jlTkH88T3Z}\ 
\begin{enumerate}[\hthree(a)]
\item There exists a low path-incompressible tree with two paths which computes a perfect path-random tree.
\item There exists an effectively proper path-random tree which does not compute any perfect path-random tree.
\end{enumerate}
\end{prop}
There are precise limits on the branching-density of path-incompressible trees. 
Given increasing $\ell=(\ell_n)$, we say
that a tree $T$ is  {\em $\ell$-perfect} if each node of length $\ell_n$ in $T$ has at least two extensions of length $\ell_{n+1}$.
\begin{thm}\label{ElRZki9oa}
If $\ell=(\ell_n)$ is computable and increasing the following are equivalent:
\begin{enumerate}[\hspace{0.3cm}(i)]
\item $\exists$ an $\ell$-perfect path-random tree
\item $\exists$ an  $\ell$-perfect path-incompressible tree 
\item $\sum_n 2^{-(\ell_{n+1}-\ell_n)}<\infty$.
\end{enumerate}
\end{thm}

We conclude with a positive answer to question (d) above: 
can a sparse perfect path-incompressible tree be effectively transformed into
a denser path-incompressible tree?
\begin{thm}[Informal]\label{VlMkthFR3B}
Any sparse computably-perfect path-incompressible tree
can be effectively transformed into an $n^2$-perfect path-incompressible tree,
almost surely.
\end{thm}
The formal statement is given in \S\ref{Cz7pIhPLAK}, but two  aspects are clear: the probabilistic success of the transformation, 
and the bound on the density of the branching  achieved.

{\bf Our methodology.} The challenge in the study of reducibilities between path-incompressible trees is 
the lack of a simple representation of the associated maps between trees.  
Our approach is based on families of {\em hitting-sets} that intersect, and {\em missing-sets} that avoid the inverse images of the maps. The intuition comes from the Fell  ({\em hit-or-miss}) topology on the space of closed sets, and the theory of measure in this space   
\citep{Molchanov}, where probability of a closed set of reals is characterized by 
the measure of its hitting-sets \citep{Axonphd, axon2015}. 

\section{Background and notation}
We lay out additional notation and known facts that we need about the Cantor space.

Strings are ordered first by length and then lexicographically.
Let 
\begin{itemize}
\item  $2^n$ denote the set of $n$-bit strings and $z\restr_n$ denote the $n$-bit prefix of $z$
\item $\twome$ denote the set of {\em reals} and  $\twomel$ the set of binary strings.
\end{itemize}
The full binary tree represents the Cantor space, with topology generated by the sets
\[
\dbra{\sigma}:=\sqbrad{z\in\twome}{\sigma\prec z}
\hspace{0.5cm}\textrm{and}\hspace{0.5cm}
\dbra{V}=\bigcup_{\sigma\in V}\dbra{\sigma}
\hspace{0.5cm}\textrm{for $V\subseteq\twomel$}.
\]
We often identify $\sigma$ with $\dbra{\sigma}$ and $V$ with $\dbra{V}$.
The uniform measure on $\twome$ is given by 
\[
\mu(\sigma):=\mu(\dbra{\sigma})=2^{-|\sigma|}
\hspace{0.3cm}\textrm{and}\hspace{0.3cm}
\mu(V):=\mu(\dbra{V}).
\]
Let $\mu_{\sigma}(V)$ be the measure of $\dbra{V}$ relative to $\dbra{\sigma}$:
\[
\mu_{\sigma}(V):=\frac{\mu(\dbra{V}\cap \dbra{\sigma})}{\mu(\sigma)}=
2^{|\sigma|}\cdot \mu(\dbra{V}\cap \dbra{\sigma}).
\]
Let $\ast$ denote the concatenation of strings and
\[
\sigma  \ast U:= \sqbrad{\sigma\ast \tau}{\tau\in U}
\hspace{0.3cm}\textrm{and}\hspace{0.3cm}
V\ast U:=\sqbrad{\sigma\ast \tau}{\sigma\in V\wedga \tau\in U}.
\]

\subsection{Background on randomness of reals}\label{JfE8m2NCZC}
A \ml test is a uniformly \ce sequence of \pf sets $V_i\subseteq\twomel$ such that $\mu(V_i)<2^{-i}$.
Following \citet{MR0223179}, we say that $x$ is random if $x\not\in \cap_i \dbra{V_i}$ for all such tests $(V_i)$.

Randomness can equivalently be defined in terms of incompressibility, via Kolmogorov complexity.
A \pf machine is a Turing machine whose domain is a \pf set of strings. 
The {\em \pf Kolmogorov complexity of $\sigma$} with respect to \pf machine $M$, denoted by $K_M(\sigma)$,
is the length of the shortest input $\rho$ such that $M(\rho)$ converges and outputs $\sigma$. There is an optimal
\pf machine $U$, such that $\sigma\mapsto K_U(\sigma)$ is minimal up to a constant, with respect to all \pf machines.
We fix an optimal \pf machine and let $K(\sigma)$ denote the corresponding \pf complexity of $\sigma$.
Given $c\geq 0$, 
we say that  $\sigma$ is {\em $c$-incompressible} if $K(\sigma)\geq |\sigma|-c$.
Schnorr showed that
a real is \ml random if and only if there is a finite upper bound  on the randomness deficiency of its initial segments.

We use the {\em Kraft-Chaitin-Levin theorem}  (see \cite[\S 3.6]{rodenisbook}), which says that  if
\begin{itemize}
\item $(\ell_{\sigma})$ is uniformly approximable from above
\item $S\subseteq\twomel$ is \ce and $\sum_{\sigma\in S} 2^{-\ell_{\sigma}}<1$
\end{itemize}
there exists a  \pf machine $M$ such that $\forall \sigma\in S,\ K_M(\sigma)\leq \ell_{\sigma}$. We use this theorem to prove the following useful Lemma.

\begin{lem}
	\label{cbce2bda}
	Let $(V_n)$ be an uniformly \ce sequence of sets of strings. If \[\sum_{\sigma\in V_n}2^{-|\sigma|}\leq 2^{-2n},\] then there is a constant $c$ such that $K(\sigma)\leq|\sigma|-n+c$ for all $\sigma\in V_n$.
\end{lem}
\begin{proof}
	We use the {\em Kraft-Chaitin-Levin theorem} by setting $S=\bigcup_n V_n$ and $l_\sigma=|\sigma|-n$ for $\sigma\in V_n$. The total weight of requests is \[\sum_{n}\sum_{\sigma\in V_{n}}2^{-(|\sigma|-n)}=\sum_n 2^n\sum_{\sigma\in V_n}2^{-|\sigma|}\leq\sum_n 2^n2^{-2n}\leq 1,\] so by the Kraft-Chaitin-Levin theorem there is a prefix-free machine $M$ such that $K_M(\sigma)\leq|\sigma|-n$ for all $\sigma\in V_n$. Then since $K$ is minimal, there is a constant $c$ such that $K(\sigma)\leq K_M(\sigma)+c$ for all $\sigma$. Then $K(\sigma)\leq |\sigma|-n+c$ for all $\sigma\in V_n$.
\end{proof}

Given a \pf $V\subseteq \twomel$ and $k$, define $V^k$ inductively by $V^1:=V$ and $V^{k+1}:=V^k\ast V$. We use the following Lemma.
\begin{lem}\label{nBTeELXDu2}
Given \pf \ce $V\subseteq\twomel$, $k>0$ and computable increasing $(t_n)$:
\begin{enumerate}[\hspace{0.3cm}(i)]
\item if $V$ contains a prefix of each nonrandom, so does $V^k$. \textup{\citep{MR820784}}
\item if $\exists c\ \forall n: K(x\restr_{t_n})>t_n-c$, then $x$ is random. 
In fact, there exists $d$ such that
\[
\forall c\ \parb{\forall n: K(x\restr_{t_n})>t_n-c}\ \Rightarrow\ 
\forall n: K(x\restr_{n})>n-c-d
\]
for all $x$, where $d$  only depends on $(t_i)$. \textup{\citep{NSTMR2140044}}
\end{enumerate}
\end{lem}

\subsection{Background on trees and randomness}\label{4UGgcdhLxh}

All trees in this paper, unless stated to be finite, are assumed to be infinite and pruned.

A string $\sigma$ \textit{branches} in $T$ if it has at least two incomparable extensions in $T$. A tree $T$ is \textit{skeletal} if there is $z\in[T]$ such that $\sigma\in T$ branches in $T$ if and only if $\sigma\prec z$. This $z$ is called the \textit{main branch} of the skeletal tree. Note that skeletal trees are proper but are neither positive nor perfect.

The \textit{$k$-prefix} of a (finite or infinite) tree $T$ is $T\upharpoonright_k=T\cap 2^{\leq k}$. Given finite tree $F$ and (finite or infinite) tree $T$, write $F\prec T$ if $F=T\upharpoonright_k$ for some $k$. In this paper, all finite trees are also assumed to be ``pruned'', in the sense that each of them is a prefix of an infinite pruned tree. For such a finite tree $F$, the \textit{height} of $F$ is the unique $k$ such that $F=T\upharpoonright_k$ for some $T$.

The standard topology on the space of all trees $\CC$ is generated by the basic open sets $\dbra{F}:=\sqbrad{T\in\CC}{F\prec T}$ where $F$ is a finite tree. In the case of $\twome$, this coincides with the hit-or-miss and the Vietoris topologies \citep[Appendix B]{Molchanov}. 
We note that $\CC$ is homeomorphic to a closed subset of the Cantor space. We only use the fact that {\em $\CC$ is compact} \citep{Molchanov}.

The {\em $\sigma$-tail} of a tree $T$ is the subtree of $T$ consisting of the nodes that  prefix or extend $\sigma$.

The following is from 
\citep[Corollary 1.12]{treeout} and, independently, \citep{deniscarlsch21}:
\begin{lem}\label{nqjFhRoRlv}
If $T$ is path-random, every positive computable tree $W$ contains a tail of $T$.
In particular, every path-random tree has a path-incompressible tail.
\end{lem}

\section{Perfect versus non-perfect path-incompressible trees}\label{P4zb649F}
We first prove Proposition \ref{jlTkH88T3Z} from \S\ref{1O8WCJrzWz}:
\begin{propa}[Proposition \ref{jlTkH88T3Z}]\ 
\begin{enumerate}[\hthree(a)]
\item There exists a low path-incompressible tree with two paths which computes a perfect path-random tree.
\item There exists an effectively proper path-random tree which does not compute any perfect path-random tree.
\end{enumerate}
\end{propa}\begin{proof}
For (a), let $z$ be a low  $\PA$ real. By \citep{BLNg08} there exist random $x,y$ with $z\equiv_{tt} x\oplus y$.
Since every $\PA$ real truth-table computes a perfect path-incompressible tree, 
the pruned tree consisting of paths $x,y$ has the required properties. 

For (b), let $z$ be random relative to $\emptyset'$ so
the pruned tree $T_z$ determined by
\[
[T_z]=\sqbrad{\tau1 \ast z}{\tau0\prec z}
\]
is effectively proper and path-random.
% and has a unique non-isolated path $z$. 
By the $\emptyset'$-randomness of $z$ and \citep{treeout},  $z$  does not compute any perfect path-random tree.
Since $T_z\equiv_{tt} z$, the same is true of $T_z$.
\end{proof}
By \citep{deniscarlsch21,treeout} every perfect path-random tree computes a path-incompressible tree. So 
Theorem \ref{dad73d1d} is a direct consequence of the following Theorem by letting $f$ be a function that dominates all computable functions.

\begin{thm}\label{4ab6fec4}
For any $f$, there is a path-incompressible proper tree $T$ computable in $\emptyset'\oplus f$, such that $T$ does not compute any path-incompressible perfect tree within \ou $f$.
\end{thm}

Toward the proof of Theorem \ref{4ab6fec4}, we set up some notations of objects that we shall use in this section. To do this, we introduce the following Lemma, which is Lemma 2.5 in \cite{codico} by setting $m_i=1$ and $L_i=l_i$. With this Lemma we are able to get a tree with fast branching speed inside any non-empty $\Pi^0_1$ class.

\begin{lem}
	\label{af272f1b}
	Let $P$ be a $\Pi^0_1$ class and $(l_i)$ be an increasing computable sequence of positive integers. If $\sum_i 2^{l_i-l_{i+1}}<\mu(P)$, then there is a $\Pi^0_1$ class $Q\subset P$ such that each string $\sigma$ of length $l_i$ with $\dbra{\sigma}\cap Q\neq\emptyset$ has at least $2$ extensions $\tau$ of length $l_{i+1}$ such that $\dbra{\tau}\cap Q\neq\emptyset$.
\end{lem}

We fix $P$ to be the set of reals with deficiency no more than $d$, where $d$ is large enough so $P$ is non-empty. We also fix $l_i=\sum_{j\leq i}(2j+c)$ where $c$ is large enough so the condition of Lemma \ref{af272f1b} is satisfied. We then fix $Q$ as guaranteed by Lemma \ref{af272f1b}.

Given a skeletal tree $T$ with main branch $z$, we say that $T$ is $(l_i)$-branching if $|\sigma|=l_i$ and $\sigma\prec z$ implies that $\sigma$ branches in $T\cap 2^{\leq l_{i+1}}$. Let $\TT$ be the class of all $(l_i)$-branching skeletal trees, which is a $\Pi^0_1$ class.

Let $\TT(Q)=\{T\in\TT:[T]\subset Q\}$ and for a finite tree $F$, let $\TT(Q,F)=\{T\in\TT(Q):F\prec T\}$. Since $Q$ is a $\Pi^0_1$ class of reals, $\TT(Q)$ and $\TT(Q,F)$ are $\Pi^0_1$ classes of trees. Lemma \ref{af272f1b} guarantees that $\TT(Q)$ is non-empty.

\subsection{Hitting cost}

We would like to use a measure on $\CC$. The set of all random reals has measure $1$ in the Cantor space, but unfortunately the set of all path-incompressible trees has uniform measure $0$ in $\CC$ \cite{treeout}. Therefore we study and use the hitting cost to serve as an alternative.

\begin{defi}
	Given a set of trees $\GG\subset\TT(Q)$ and a set of finite strings $H$, we say that \textit{$H$ hits $\GG$} if $T\cap H\neq\emptyset$ for any $T\in\GG$. The \textit{hitting-cost} of $\GG$ relative to $Q$, denoted by $c_Q(\GG)$, is $\inf\mu(\dbra{H}\cap Q)$ over all $H$ that hits $\GG$.
\end{defi}

Observe that $T\cap H\neq\emptyset$ if and only if $[T]\cap\dbra{H}\neq\emptyset$, so if $\dbra{H_1}=\dbra{H_2}$ then $H_1$ hits $\GG$ if and only if $H_2$ hits $\GG$.
We then use compactness to show that when working with a closed set of trees, we can take everything to be finite.

\begin{lem}
	\label{13958859}
	If $\GG\subset\TT(Q)$ is a closed set of trees and $H$ hits $\GG$, then there is a finite subset of $H$ that hits a clopen superset of $\GG$.
\end{lem}
\begin{proof}
	We first prove that there is a finite subset of $H$ that hits $\GG$. $\GG$ is a closed set in a compact space so it is also compact. Then note that
	\begin{align*}
		H\text{ hits }\GG
		&\iff \forall T\in\GG,H\cap T\neq\emptyset \\
		&\iff \GG\subset\{T:H\cap T\neq\emptyset\} \\
		&\iff \GG\subset\bigcup_{\sigma\in H}\{T:\sigma\in T\}
	\end{align*}
	and that each $\{T:\sigma\in T\}$ is clopen. Now $\bigcup_{\sigma\in H}\{T:\sigma\in T\}$ is an open cover of $\GG$, so it contains a finite cover of $\GG$. The indices of this finite cover is a finite subset of $H$ that hits $\GG$.

	We then assume that $H$ itself is finite, and prove that it hits a clopen superset of $\GG$. Being a closed set, let $\GG=\CC-\bigcup_{i\in I}{\dbra{F_i}}$ where $F_i$ are finite trees. Then
	\begin{align*}
		H\text{ hits }\GG
		&\iff \GG\subset\bigcup_{\sigma\in H}\{T:\sigma\in T\} \\
		&\iff \CC-\bigcup_{i\in I}{\dbra{F_i}}\subset\bigcup_{\sigma\in H}\{T:\sigma\in T\} \\
		&\iff \CC-\bigcup_{\sigma\in H}\{T:\sigma\in T\}\subset\bigcup_{i\in I}{\dbra{F_i}}
	\end{align*}
	and each $\dbra{F_i}$ is also clopen. Now $\bigcup_{i\in I}{\dbra{F_i}}$ is an open cover of $\CC-\bigcup_{\sigma\in H}\{T:\sigma\in T\}$, which is clopen thus compact since $H$ is finite, so again $\bigcup_{i\in I}{\dbra{F_i}}$ contains a finite cover $\bigcup_{i\leq n}{\dbra{F_i}}$. Then $H$ hits $\CC-\bigcup_{i\leq n}{\dbra{F_i}}$ which is a clopen superset of $\GG$.
\end{proof}

\begin{lem}
	\label{0240295d}
	Let $\GG\subset\TT(Q)$ be a $\Pi^0_1$ class of trees. If $c_Q(\GG)=0$ then $\GG=\emptyset$.
\end{lem}
\begin{proof}
	For each $n$, since $c_Q(\GG)=0$ there is $H_n$ with $\mu(\dbra{H_n}\cap Q)\leq 2^{-2n-1}$ that hits $\GG$. By Lemma \ref{13958859}, there is a finite such $H_n$ that hits a clopen superset of $\GG$. Since $Q$ is $\Pi^0_1$, there is some stage $s_n$ in the approximation of $Q$ where $\mu(\dbra{H_n}\cap Q_{s_n})\leq 2^{-2n}$.
	
	We can effectively check if a finite set of strings hits a clopen set of trees, and the set of all clopen supersets of $\GG$ is \ce since $\GG$ is $\Pi^0_1$. Also given finite $H_n$ and $s_n$, we can effectively check if $\mu(\dbra{H_n}\cap Q_{s_n})\leq 2^{-2n}$. Therefore given $n$ we can effectively search for such an $H_n$ and $s_n$. In this way we get a computable sequence of $(H_n)$ with $(s_n)$ such that $\mu(\dbra{H_n}\cap Q_{s_n})\leq 2^{-2n}$ and that each $H_n$ hits $\GG$. Both $\dbra{H_n}$ and $Q_{s_n}$ are clopen, we let $V_n$ be finite and prefix-free such that $\dbra{V_n}=\dbra{H_n}\cap Q_{s_n}$, then $(V_n)$ is uniformly computable. By Lemma \ref{cbce2bda}, there is a constant $c$ such that $K(\sigma)\leq |\sigma|-n+c$ for all $\sigma\in V_n$, \ie each string in $V_n$ has deficiency at least $n-c$.
	
	Now suppose $\GG\neq\emptyset$ and take any $T\in\GG$. For each $n$, $H_n$ hits $\GG$ so $\dbra{H_n}\cap[T]\neq\emptyset$. Take any $x\in\dbra{H_n}\cap[T]$. As $[T]\subset Q\subset Q_{s_n}$, $x$ is also in $\dbra{H_n}\cap Q_{s_n}=\dbra{V_n}$. So $x$ has a prefix in $V_n$ with deficiency at least $n-c$. But $x$ is also in $Q$, so this prefix has deficiency at most $d$. Taking $n=d+c+1$ leads to a contradiction. Therefore $\GG=\emptyset$.
\end{proof}

\subsection{Tree-functionals}

We now clarify how does one tree compute another.

\begin{defi}
	A \textit{tree functional} is a Turing functional $\Phi$ that takes tree $T$ as an oracle and $n\in\omega$ as an input, and outputs a non-empty finite tree $\Phi(T,n)$, such that if $\Phi(T,n-1)$ and $\Phi(T,n)$ both halt, then
	\begin{itemize}
		\item $\Phi(T,n-1)\prec\Phi(T,n)$,
		\item each $\sigma\in\Phi(T,n-1)$ branches in $\Phi(T,n)$.
	\end{itemize}
	Let $\Phi(T)=\bigcup_n\Phi(T,n)$.
\end{defi}

Now if $\Phi(T,n)$ halts for all $n$, we say that $\Phi$ is total on $T$. In this case, $\Phi(T)$ is an infinite tree. The second condition also makes it automatically perfect. By induction, the second condition also implies that $\Phi(T,n)$ has height at least $n$.

There is a computable enumeration of indices of Turing functionals, such that each index is an index of a tree functional, and that at least one index for each tree functional is enumerated. Let $\Phi_e$ be the $e^{th}$ tree functional in this enumeration.

The \textit{oracle-use} of $\Phi(T,n)$ is the maximum $k$ such that $T\upharpoonright_k$ is accessed in the computation of $\Phi(T,n)$. By ``$\Phi(T,n)$ halts within \ou $k$'' we mean that $\Phi(T,n)$ halts and the \ou of $\Phi(T,n)$ is no more than $k$. Given a function $f:\omega\to\omega$, by ``$\Phi$ is total on $T$ within \ou $f$'' we mean that $\Phi(T,n)$ halts within \ou $f(n)$ for all $n$.

Given a string $\sigma$ and a tree functional $\Phi$, let $\Phi^{-1}(\sigma)=\{T:\sigma\in\Phi(T)\}$. Clearly $\Phi^{-1}(\sigma)$ is $\Sigma^0_1$, and the next Lemma shows that it is indeed clopen inside any closed set where $\Phi$ is total within bounded \ou.

\begin{lem}
	\label{a42e15c4}
	Let $\GG\subset\TT(Q)$ be a closed set of trees, $f$ a function and $\Phi$ a tree functional such that $\Phi$ is total on all $T\in\GG$ within \ou $f$. Then $\GG\cap\Phi^{-1}(\sigma)=\GG\cap\bigcup_i\dbra{F_i}$ where $(F_i)$ are finitely many trees with height $f(|\sigma|)$. In particular, $\GG\cap\Phi^{-1}(\sigma)$ is closed.
\end{lem}
\begin{proof}
	Take any $T\in\GG\cap\Phi^{-1}(\sigma)$. Since $\sigma\in\Phi(T)$ and $\Phi(T,|\sigma|)$ contains a string with length $|\sigma|$ and is a prefix of $\Phi(T)$, we have $\sigma\in\Phi(T,|\sigma|)$. The \ou of $\Phi(T,|\sigma|)$ is no more than $f(|\sigma|)$, so $\sigma\in\Phi(T')$ for any $T'\succ T\upharpoonright_{f(|\sigma|)}$, \ie $\dbra{T\upharpoonright_{f(|\sigma|)}}\subset\Phi^{-1}(\sigma)$. By taking the union over all $T$ in $\GG\cap\Phi^{-1}(\sigma)$, we get \[\GG\cap\Phi^{-1}(\sigma)\subset\bigcup_{T\in\GG\cap\Phi^{-1}(\sigma)}\dbra{T\upharpoonright_{f(|\sigma|)}}\subset\Phi^{-1}(\sigma).\] Now $T\upharpoonright_{f(|\sigma|)}$ is a finite tree and there are only finitely many possibilities for such a tree, so the middle term is $\bigcup_i\dbra{F_i}$ where $(F_i)$ are finitely many trees with height $f(|\sigma|)$. By intersecting $\GG$ with all sets in the above inclusion, we get $\GG\cap\Phi^{-1}(\sigma)=\GG\cap\bigcup_i\dbra{F_i}$, as desired.
\end{proof}

\begin{lem}
	\label{1604aefc}
	Let $F$ be a finite tree with height $h$, $f$ a function and $\Phi$ a tree functional such that $\Phi$ is total on all $T\in\TT(Q,F)$ within \ou $f$. Let $n$ be such that $l_n\geq h$ and that $l_n\geq f(|\sigma|)$. If $H'\subset 2^{l_{n+1}}$ hits $\TT(Q,F)\cap\Phi^{-1}(\sigma)$, then there is $H\subset 2^{l_n}$ with \[\mu((\dbra{H}-\dbra{H'})\cap Q)\leq 2^{l_n-l_{n+1}}\] that also hits $\TT(Q,F)\cap\Phi^{-1}(\sigma)$.
\end{lem}
\begin{proof}
Let $H=\{\tau\in 2^{l_n}:\mu((\dbra{\tau}-\dbra{H'})\cap Q)\leq 2^{-l_{n+1}}\}$ so $H\subset 2^{l_n}$ and \[\mu((\dbra{H}-\dbra{H'})\cap Q)\leq|H|\cdot 2^{-l_{n+1}}\leq 2^{l_n-l_{n+1}}.\] To prove that $H$ hits $\TT(Q,F)\cap\Phi^{-1}(\sigma)$, suppose otherwise there is $T\in\TT(Q,F)\cap\Phi^{-1}(\sigma)$ such that $T\cap H=\emptyset$. Then
	\begin{align*}
		\tau\in T\text{ and }|\tau|=l_n
		&\implies \tau\notin H\text{ and }|\tau|=l_n \\
		&\implies \mu((\dbra{\tau}-\dbra{H'})\cap Q)>2^{-l_{n+1}} \\
		&\implies \text{ there is }x_1,x_2\in(\dbra{\tau}-\dbra{H'})\cap Q\text{ with }x_1\upharpoonright_{l_{n+1}}\neq x_2\upharpoonright_{l_{n+1}}.
	\end{align*}

	Take $\sigma_1=x_1\upharpoonright_{l_{n+1}}$ and $\sigma_2=x_2\upharpoonright_{l_{n+1}}$, then $\sigma_1\neq\sigma_2$ and we shall claim some properties on them. Firstly, since $\dbra{\sigma_1}\cap\dbra{\tau}\neq\emptyset$ and $|\tau|=l_n<l_{n+1}=|\sigma_1|$, we have $\tau\prec\sigma_1$. Also, as $\dbra{\sigma_1}-\dbra{H'}\neq\emptyset$, with the fact that $|\sigma_1|=l_{n+1}$ and $H'\subset 2^{l_{n+1}}$ we have $\sigma_1\notin H'$, so $\dbra{\sigma_1}\cap\dbra{H'}=\emptyset$. Finally, together with $\dbra{\sigma_1}\cap Q\neq\emptyset$ we have $\dbra{\sigma_1}\cap(Q-\dbra{H'})\neq\emptyset$. The same claims hold for $\sigma_2$.

	We then build an $(l_i)$-branching skeletal tree $T'$. We start from $T'=T\upharpoonright_{l_n}$ and let $z$ be the main branch of $T$. For each $\tau\in T$ and $|\tau|=l_n$, we take $\sigma_1$ and $\sigma_2$ extending $\tau$ as in the previous argument. We first use Lemma \ref{af272f1b} to extend $\sigma_1$ to a real in $Q-\dbra{H'}$, and put all initial segments of the real in $T'$. Additionally, if $\tau=z\upharpoonright_{l_n}$, we also use Lemma \ref{af272f1b} to extend $\sigma_2$ to an $(l_i)_{i\geq n+1}$-branching skeletal tree whose paths are in $Q-\dbra{H'}$. We add the strings in this skeletal tree into $T'$.
	
	We can verify that $T'$ is indeed $(l_i)$-branching. Additionally $T'\succ T\upharpoonright_{l_n}\succ F$ and $[T']\subset Q$, so $T'\in\TT(Q,F)$. Since $\dbra{T\upharpoonright_{l_n}}\subset\Phi^{-1}(\sigma)$ we also have $T'\in\Phi^{-1}(\sigma)$. But $\dbra{H'}\cap[T']=\emptyset$, contradicting the fact that $H'$ hits $\TT(Q,F)\cap\Phi^{-1}(\sigma)$.
\end{proof}

\subsection{Envelopes}
In this subsection, we fix a finite tree $F$ with height $h$, a function $f$ and a tree functional $\Phi$ such that $\Phi$ is total on all $T\in\TT(Q,F)$ within \ou $f$.

This subsection is devoted to establish the fact that if the hitting cost of $\TT(Q,F)\cap\Phi^{-1}(\sigma)$ is small, then $c_Q(\TT(Q,F))=0$. We split this task into 3 steps. Firstly, if the hitting cost of $\TT(Q,F)\cap\Phi^{-1}(\sigma)$ is small, then a well-behaved finite family (which we shall call a finite envelope) hits them. Then we use compactness argument to extend this family to an infinite one. Finally we use this infinite family to generate sets that hit $\TT(Q,F)$, showing that $c_Q(\TT(Q,F))=0$.

We let $l_\sigma$ be the first $l_n$ in $(l_i)$ such that $l_n\geq h$, $l_n\geq f(|\sigma|)$ and that $l_n>l_\tau$ where $\tau$ is any proper prefix of $\sigma$. We note that $l_\sigma$ is only determined by $|\sigma|$, not $\sigma$ itself. Also the last requirement implies that $n\geq|\sigma|$.

\begin{defi}
	An \textit{$r$-envelope} is a family $(H_\sigma)$ indexed by $\sigma\in 2^{<\omega}$ such that for all $\sigma$,
	\begin{enumerate}
		\item $\mu(\dbra{H_\sigma}\cap Q)\leq 2^{r-|\sigma|}+\sum_{l_\sigma\leq l_i}2^{l_i-l_{i+1}}$,
		\item $H_\sigma$ hits $\TT(Q,F)\cap\Phi^{-1}(\sigma)$,
		\item $H_\sigma\subset 2^{l_\sigma}$,
		\item $\mu((\dbra{H_\sigma}-\dbra{H_{\sigma 0}}-\dbra{H_{\sigma 1}})\cap Q)\leq\sum_{l_\sigma\leq l_i<l_{\sigma 0}}2^{l_i-l_{i+1}}$.
	\end{enumerate}
	A \textit{finite $r$-envelope} is defined similarly, except that it is indexed by $\sigma\in 2^{\leq n}$ for some $n$, which is referred to as the \textit{length} of the envelope.
\end{defi}

\begin{lem}
	\label{96cab15f}
	If there is a constant $r$ such that $c_Q(\TT(Q,F)\cap\Phi^{-1}(\sigma))\leq 2^{r-|\sigma|}$ for all $\sigma$, then there are finite $r$-envelopes of any length $n$.
\end{lem}
\begin{proof}
	For each $|\sigma|=n$, there is $H_\sigma'$ with $\mu(\dbra{H_\sigma'}\cap Q)\leq 2^{r-|\sigma|}$ that hits $\TT(Q,F)\cap\Phi^{-1}(\sigma)$. By Lemma \ref{a42e15c4}, $\TT(Q,F)\cap\Phi^{-1}(\sigma)$ is closed, so by Lemma \ref{13958859} we can choose this $H_\sigma'$ to be finite. Let $L$ be in $(l_i)$ and longer than any string in this finite $H_\sigma'$. Then $\{\tau\in 2^L:\dbra{\tau}\subset\dbra{H'_\sigma}\}$ has the same measure as $H_\sigma'$ and also hits $\TT(Q,F)\cap\Phi^{-1}(\sigma)$. By repeatedly applying Lemma \ref{1604aefc}, there is $H_\sigma\subset 2^{l_\sigma}$ with \[\mu(\dbra{H_\sigma}\cap Q)\leq\mu(\dbra{H_\sigma'}\cap Q)+\sum_{l_\sigma\leq l_i<L}2^{l_i-l_{i+1}}\] that also hits $\TT(Q,F)\cap\Phi^{-1}(\sigma)$. Thus the family $H_\sigma$ satisfies conditions 1, 2 and 3.

	We extend these $H_\sigma$ to a finite envelope of length $n$. Inductively for all $|\sigma|<n$ (from $|\sigma|=n-1$ to $|\sigma|=0$), let $H$ be such that $\dbra{H}=\dbra{H_{\sigma 0}}\cup\dbra{H_{\sigma 1}}$.

	We claim that $H$ hits $\TT(Q,F)\cap\Phi^{-1}(\sigma)$. To see this, observe that if $T\in\TT(Q,F)\cap\Phi^{-1}(\sigma)$, then since $\sigma\in\Phi(T)$ and $\Phi$ is total on $T$, at least one of $\sigma 0$ or $\sigma 1$ is in $\Phi(T)$. So \[\TT(Q,F)\cap\Phi^{-1}(\sigma)\subset(\TT(Q,F)\cap\Phi^{-1}(\sigma 0))\cup(\TT(Q,F)\cap\Phi^{-1}(\sigma 1)).\] Therefore by the fact that $H_{\sigma 0}$ hits $\TT(Q,F)\cap\Phi^{-1}(\sigma 0)$ and $H_{\sigma 1}$ hits $\TT(Q,F)\cap\Phi^{-1}(\sigma 1)$, we have that $H_{\sigma 0}\cup H_{\sigma 1}$ hits $\TT(Q,F)\cap\Phi^{-1}(\sigma)$, and so does $H$.

	Again by repeatedly applying Lemma \ref{1604aefc}, there is $H_\sigma\subset 2^{l_\sigma}$ with \[\mu((\dbra{H_\sigma}-\dbra{H})\cap Q)\leq\sum_{l_\sigma\leq l_i<l_{\sigma 0}}2^{l_i-l_{i+1}}\] that also hits $\TT(Q,F)\cap\Phi^{-1}(\sigma)$. Thus conditions 2, 3 and 4 are satisfied. Condition 1 is satisfied for $|\sigma|=n$, and for $|\sigma|<n$ it is automatically implied by condition 4.
\end{proof}

\begin{lem}
	\label{7c902466}
	If there are finite $r$-envelopes for any length $n$, then there is an infinite $r$-envelope.
\end{lem}
\begin{proof}
	By the assumption, for each $n$ let $\Hs_n=(H_{n,\sigma})$ be a finite $r$-envelope of length $n$. We define another sequence of finite $r$-envelopes $\Hs'_n$ of length $n$ such that $\Hs'_{n-1}$ is the restriction of $\Hs'_n$ to $\sigma\in 2^{\leq n-1}$. Then $\lim\Hs'_n$ is an infinite $r$-envelope.
	
	To choose this sequence, we start from $\Hs'_0$ being the empty envelope of length $0$. In each stage $n$, find $\Hs'_n$ extending $\Hs'_{n-1}$ such that $\Hs'_n$ is the restriction of $\Hs_m$ for infinitely many $m$. Such an extension can always be found since there are only finitely many choices for a finite envelope of a certain length, but infinitely many restrictions of $\Hs_m$.
\end{proof}

Before proceeding, we simplify the term in condition 1 of an envelope. If $l_\sigma=l_n$ then \[\sum_{l_\sigma\leq l_i}2^{l_i-l_{i+1}}=\sum_{i=n}^\infty 2^{-l_{i+1}}=\sum_{i=n}^\infty 2^{-2i-c-1}=\frac{2^{-2n-c+1}}{3}\leq 2^{-|\sigma|-c}\leq 2^{r-|\sigma|},\] therefore condition 1 implies that $\mu(\dbra{H_\sigma}\cap Q)\leq 2^{1+r-|\sigma|}$.

\begin{lem}
	\label{2d7bbb1e}
	If there is an infinite $r$-envelope $(H_\sigma)$ then $c_Q(\TT(Q,F))=0$.
\end{lem}
\begin{proof}
	For any real $z$, let $\hit(z,n)=|\{\sigma\in 2^n:z\in\dbra{H_\sigma}\}|$. 
	
	As $\dbra{H_\sigma}$ could contain more than $\dbra{H_{\sigma 0}}\cup\dbra{H_{\sigma 1}}$, $\hit(z,n)$ is generally not non-decreasing in $n$. However, by condition 4 of the envelope, the measure of all $z\in Q$ such that $\hit(z,n)$ is decreasing at $n=|\sigma|$ is no more than $2^n\sum_{l_\sigma\leq l_i<l_{\sigma 0}}2^{l_i-l_{i+1}}$. Fix any $m$. Let $V_m$ be the set of all $z$ such that $\hit(z,n)$ is decreasing at any $n\geq m$, then \[\mu(V_m\cap Q)\leq\sum_{n\geq m}2^n\sum_{l_\sigma\leq l_i<l_{\sigma 0}}2^{l_i-l_{i+1}}\leq \sum_{i\geq m}2^{-2i-c-1}\sum_{n\leq i}2^n\leq 2^{1-m-c}.\]
	
	For $z\notin V_m$, $\hit(z,n)$ is non-decreasing on $n\geq m$, so we can set $\hit(z)=\lim_n\hit(z,n)$. Let \[U^n_k=\{z:\hit(z,n)\geq 2^k\}\text{ and }U_{k,m}=\bigcup_{n\geq m}U^n_k.\] For a particular $n$, consider the sum \[\sum_{\sigma\in 2^n}\mu(\dbra{H_\sigma}\cap Q)\leq 2^{1+r}\] and note that each set of reals contributing its measure to $U^n_k\cap Q$ also contributes $2^k$ times its measure to the above sum, so $\mu(U^n_k\cap Q)\leq 2^{1+r-k}$. Note that $U^n_k-V_m$ is non-decreasing for $n\geq m$, so \[\mu((U_{k,m}-V_m)\cap Q)\leq\lim_n\mu(U^n_k\cap Q)\leq 2^{1+r-k}.\]

	Note that $V_m$ and $U_{k,m}$ are open. Let $H$ be finite, prefix-free and such that $\dbra{H}=V_m\cup U_{k,m}$, we claim that $H$ hits $\TT(Q,F)$, so \[c_Q(\TT(Q,F))\leq\mu((V_m\cup U_{k,m})\cap Q)\leq 2^{1-m-c}+2^{1+r-k},\] and since $m$ and $k$ are arbitrary we have $c_Q(\TT(Q,F))=0$. To prove the claim, suppose otherwise there is $T\in\TT(Q,F)$ such that $[T]\cap V_m=\emptyset$ and $[T]\cap U_{k,m}=\emptyset$. Then for any $z\in[T]$, $\hit(z)$ is defined, so $\hit(z,n)$ is bounded.

	$T$ is a skeletal tree, let $z_0$ be its main branch. $\Phi(T)$ is perfect and $\hit(z_0,n)$ is bounded, so there is $\sigma_0\in\Phi(T)$ with $z\notin\dbra{H_{\sigma_0}}$. Since $H_{\sigma_0}$ is finite, there is $\tau\prec z$ such that $\dbra{\tau}\cap\dbra{H_{\sigma_0}}=\emptyset$.

	There are only finitely many paths $z_1,\cdots,z_n$ in $[T]$ not prefixed by $\tau$. Inductively in $k$ (from $1$ to $n$), since $\hit(z_k,n)$ is bounded, there is $\sigma_k\in\Phi(T)$ with $\sigma_k\succ\sigma_{k-1}$ and $z_k\notin\dbra{H_{\sigma_k}}$. Let $\sigma=\sigma_n$. Now $\dbra{H_{\sigma_0}}\supset\dbra{H_{\sigma_1}}\supset\cdots\supset\dbra{H_{\sigma_n}}=\dbra{H_{\sigma}}$.

	For a path in $[T]$ prefixed by $\tau$, since $\dbra{\tau}\cap\dbra{H_{\sigma_0}}=\emptyset$ it is not in $\dbra{H_\sigma}$. For a path $z_k$ in $[T]$ not prefixed by $\tau$, since $z_k\notin\dbra{H_{\sigma_k}}$, it is not in $\dbra{H_\sigma}$. So $[T]\cap \dbra{H_\sigma}=\emptyset$. But $H_\sigma$ hits $\TT(Q,F)\cap\Phi^{-1}(\sigma)$, and indeed $\sigma\in\Phi(T)$ so $T\in\Phi^{-1}(\sigma)$, a contradiction.
\end{proof}

\subsection{Proof of Theorem \ref{4ab6fec4}}

Combining Lemma \ref{96cab15f}, Lemma \ref{7c902466}, Lemma \ref{2d7bbb1e} and Lemma \ref{0240295d}, we get the following.

\begin{lem}
	\label{e46c53e5}
	Let $F$ be a finite tree, $f$ a function and $\Phi$ a tree functional such that $\Phi$ is total on all $T\in\TT(Q,F)$ within \ou $f$. If there is a constant $r$ such that $c_Q(\TT(Q,F)\cap\Phi^{-1}(\sigma))\leq 2^{r-|\sigma|}$ for all $\sigma$, then $\TT(Q,F)=\emptyset$.
\end{lem}

To prove Theorem \ref{4ab6fec4} we need a final Lemma.

\begin{lem}
	\label{59679e46}
	Let $F$ be a finite tree, $f$ a function and $\Phi$ a tree functional such that $\Phi$ is total on all $T\in\TT(Q,F)$ within \ou $f$. If there is no constant $r$ such that $c_Q(\TT(Q,F)\cap\Phi^{-1}(\sigma))\leq 2^{r-|\sigma|}$ for each $\sigma$, then for each $m$, there is some $T\in\TT(Q,F)$ such that $\Phi(T)$ has deficiency at least $m$.
\end{lem}
\begin{proof}
	For each $n$, enumerate $V_n$ by the following loop, starting with $s=0$.

	\begin{enumerate}
		\item Find $t$ and $\sigma$ with $t>s$ and $|\sigma|>2n$, such that $\dbra{\sigma}\cap V_{n,s}=\emptyset$ and \[c_{Q_s}(\TT(Q_s,F)\cap\Phi^{-1}_t(\sigma))>2^{2n+1-|\sigma|}.\] Enumerate the least such $\sigma$ into $V_n$.
		\item Wait for a stage $r$ such that $\TT(Q_r,F)\cap\Phi^{-1}_t(\sigma)=\emptyset$. Set $s:=r$ and go back to step 1.
	\end{enumerate}

	Consider a single loop. From $\TT(Q_s,F)\cap\Phi^{-1}_t(\sigma)$ to $\TT(Q_r,F)\cap\Phi^{-1}_t(\sigma)=\emptyset$ only trees that have a path in $Q_s-Q_r$ are removed, so $Q_s-Q_r$ hits $\TT(Q_s,F)\cap\Phi^{-1}_t(\sigma)$. Therefore \[\mu(Q_s-Q_r)=\mu((Q_s-Q_r)\cap Q_s)>2^{2n+1-|\sigma|},\] so $2^{-|\sigma|}\leq 2^{-2n-1}\mu(Q_s-Q_r)$. Taking the sum over all loops, taking into account that the last loop could be incomplete (where $r$ is never found) and that $|\sigma|>2n$, we have \[\sum_{\sigma\in V_n}2^{-|\sigma|}\leq 2^{-2n-1}+2^{-2n-1}=2^{-2n}.\]

	Now $V_n$ is uniformly c.e., using Lemma \ref{cbce2bda} there is a constant $c$ such that $K(\sigma)\leq|\sigma|-n+c$ for all $\sigma\in V_n$. Then $K(\sigma)\leq|\sigma|-m$ for all $\sigma\in V_{m+c}$, \ie each string in $V_{m+c}$ has deficiency at least $m$.

	There is $\sigma$ with $|\sigma|>2n$ such that $c_Q(\TT(Q,F)\cap\Phi^{-1}(\sigma))>2^{2(m+c)+1-|\sigma|}$, since otherwise by taking $r$ be the maximum of $2(m+c)+1$ or the constant needed for all $|\sigma|\leq 2n$, we have $c_Q(\TT(Q,F)\cap\Phi^{-1}(\sigma))\leq 2^{r-|\sigma|}$ which contradicts the condition.
	
	As $Q_s\supset Q$ we have $c_{Q_s}\geq c_Q$, so $c_{Q_s}(\TT(Q,F)\cap\Phi^{-1}(\sigma))>2^{2(m+c)+1-|\sigma|}$. Eventually such a $\sigma$ is enumerated into $V_{m+c}$, and then eventually the construction of $V_{m+c}$ gets stuck at step 2. Therefore $\TT(Q,F)\cap\Phi^{-1}(\sigma)\neq\emptyset$, \ie there is $T\in\TT(Q,F)$ with $\sigma\in\Phi(T)$. With $\sigma\in V_{m+c}$, $\Phi(T)$ has deficiency at least $m$, as desired.
\end{proof}

Finally we can prove Theorem \ref{4ab6fec4}.

\begin{proof}[Proof of Theorem \ref{4ab6fec4}]
	We build the tree from $F_{-1}=\{\lambda\}$ which is extendable in $\TT(Q)$, and for each index $e$ and constant $m$, we extend the tree $F_{\brag{e,m}-1}$ to some $F_\brag{e,m}$ that is extendable in $\TT(Q)$, satisfying
	\begin{center}
		$R_{e,m}$: for any tree $T\succ F_\brag{e,m}$, if $\Phi_e$ is total on $T$ within \ou $f$, then $\Phi_e(T)$ has deficiency at least $m$.
	\end{center}
	Finally we let $T=\bigcup_i F_i$. Then for any $e$, if $\Phi_e$ is total on $T$ within \ou $f$, then $\Phi_e(T)$ has infinite deficiency. Thus this $T$ does not compute any path-incompressible perfect tree within \ou $f$, as desired.

	Now we show how to extend $F_{\brag{e,m}-1}$ to some $F_\brag{e,m}$ extendable in $\TT(Q)$, satisfying $R_{e,m}$. For simplicity write $F=F_{\brag{e,m}-1}$. One of the following two cases happens.
	\begin{itemize}
		\item $\Phi_e$ is not total on some $T\in\TT(Q,F)$ within \ou $f$.
		
		There is some $n$ such that $\Phi_e(T,n)$ does not halt within \ou $f(n)$. Let $F'=T\upharpoonright_{f(n)}$, then $F\prec F'\prec T$ and for any $T'\succ F'$, $\Phi_e(T')$ does not halt within \ou $f(n)$. Therefore there is $n$ and $F'$ such that
		\begin{enumerate}
			\item $F'\succ F$ and $F'$ is extendable in $\TT(Q)$, \ie $\TT(Q,F')\neq\emptyset$,
			\item $\Phi_e(F',n)$ uses oracle no more than $F'$ and does not halt.
		\end{enumerate}

		Using oracle $\emptyset'$ and $f$, we can effectively check if $n$ and $F'$ witness that this case happens. If we found such $n$ and $F'$, let $F_\brag{e,m}=F'$, then for any $T\succ F_\brag{e,m}$, $\Phi_e(T)$ is not total within \ou $f$.

		\item $\Phi_e$ is total on all $T\in\TT(Q,F)$ within \ou $f$.
		
		By Lemma \ref{e46c53e5}, since $\TT(Q,F)$ is non-empty by the previous steps of the construction, there is no constant $r$ such that $c_Q(\TT(Q,F)\cap\Phi_e^{-1}(\sigma))\leq 2^{r-|\sigma|}$. Then by Lemma \ref{59679e46}, there is some $T\in\TT(Q,F)$ such that $\Phi_e(T)$ has deficiency at least $m$, \ie it contains some $\sigma$ with deficiency at least $m$. Then again there is some $n$ such that $\sigma\in\Phi_e(T,n)$, and similarly there is $n$ and $F'$ such that
		\begin{enumerate}
			\item $F'\succ F$ and $F'$ is extendable in $\TT(Q)$, \ie $\TT(Q,F')\neq\emptyset$,
			\item $\Phi_e(F',n)$ uses oracle no more than $F'$ and has deficiency at least $m$.
		\end{enumerate}

		Also using oracle $\emptyset'$, we can effectively check if $n$ and $F'$ witness that this case happens. If we found such $n$ and $F'$, let $F_\brag{e,m}=F'$, then for any $T\succ F_\brag{e,m}$, $\Phi_e(T)$ has deficiency at least $m$.
	\end{itemize}
\end{proof}

\section{Perfect path-incompressible trees}\label{Cz7pIhPLAK}
We examine the
density of branching in path-incompressible trees,
and the possibility of effectively increasing it.
The density of branching can be formalized in terms of a computable increasing sequence $(\ell_n)$ as follows.
\begin{defi}\label{BO79cnNwnp}
Given computable increasing $\ell=(\ell_n)$, tree $T$, $z\in [T]$ we say that
\begin{itemize}
\item $T$ is {\em $\ell$-perfect} if for almost all 
$n$, each $\sigma\in T\cap 2^{\ell_n}$ has $\geq 2$ extensions in $T\cap 2^{\ell_{n+1}}$
\item $z$ is {\em $(\ell, T)$-branching} if for almost all  $n$, 
each $z\restr_{\ell_n}$ has $\geq 2$ extensions in $T\cap 2^{\ell_{n+1}}$,
\end{itemize}
where {\em almost all $n$} means all but finitely many $n$.
\end{defi}
We use the following  assembly of items from \citep{codico}:
\begin{lem}\label{2PimWgH6lS}
Let $\ell=(\ell_i)$ be computable  and increasing, and 
$P$ be a positive \pz tree with $[P]$ consisting entirely of randoms. 
\begin{enumerate}[\hspace{0.3cm}(i)]
\item If $\sum_i 2^{-(\ell_{i+1}-\ell_{i})}=\infty$, every $(\ell,P)$-branching  $z$ is incomplete; also there are
arbitrarily large $n$ and $\sigma\in 2^{\ell_n}\cap P$ which have exactly one extension in $2^{\ell_{n+1}}\cap P$.
\item If $\sum_i 2^{-(\ell_{i+1}-\ell_{i})}<\infty$,  there exists an  $\ell$-perfect tree $T\subseteq P$ and an injection $z\mapsto f(z)\in [T]$ such that $z\leq_T f(z)$.
\end{enumerate}
\end{lem}\begin{proof}
The first part of (i) is by \citep[Lemma 2.8]{codico},
taking into account that a random is difference-random iff it does not compute $\zp$.
The second part of (i) is by \citep[Lemma 2.4 \& Corollary 2.9]{codico}.
Clause (ii) is  \citep[Lemma 2.2]{codico}.
\end{proof}

\subsection{Density of branching in path-incompressible trees}\label{qhzXZfYv4}
We characterize the density of branching that a 
perfect path-random tree can have:
\begin{thm}\label{KmILdNfsur}
Given computable increasing $\ell=(\ell_n)$, the following are equivalent:
\begin{enumerate}[\hspace{0.3cm}(a)]
\item $\exists$ an $\ell$-perfect path-random tree
\item $\exists$ an $\ell$-perfect path-incompressible tree 
\item $\sum_n 2^{-(\ell_{n+1}-\ell_n)}<\infty$.
\end{enumerate}
\end{thm}\begin{proof}
Implication (b)$\to$(a) is trivial, while
(c)$\to$(b) follows from  Lemma \ref{2PimWgH6lS} (ii).

By Lemma \ref{nqjFhRoRlv} we get (a)$\to$(b).
For $\neg$(c)$\to$$\neg$(b), let 
$\FF_{\ell}$ be the class of trees $T$ such that each
$\sigma\in T\cap 2^{\ell_n}$ has exactly two extensions in $T\cap 2^{\ell_{n+1}}$.
Let $P$ be a  \pz pruned tree $P$ of finite deficiency and
let $\FF_{\ell}(P)$ be the restriction of $\FF_{\ell}$ to trees $T$ with $T\subseteq P$.
Then
\begin{enumerate}[\hspace{0.3cm}(I)]
\item $Q:=\sqbrad{\sigma}{\exists T\in \FF_{\ell}(P),\ \sigma\in T}$ is a \pz subtree of $P$, by compactness.
\item each $\sigma\in Q\cap 2^{\ell_n}$ has at least two extensions in $Q\cap 2^{\ell_{n+1}}$.
\end{enumerate}
If $\FF_{\ell}(P)\neq\emptyset$ then $Q$ is infinite; so
for $\neg$(c)$\to$$\neg$(b) it remains to show that $Q$ is finite. 
Assuming otherwise,
by $\neg$(c) and the second clause of Lemma \ref{2PimWgH6lS}(i), 
there exists $n$ and $\sigma\in Q\cap 2^{\ell_n}$ which has at most 
one extension in $2^{\ell_{n+1}}\cap Q$. But this contradicts (II) above. 
\end{proof}
Next, we consider the branching density along a path in a path-incompressible \pz tree $P$.
It is known that Turing-hard members of $P$ have low density in $P$, hence sparse branching.

\begin{defi}
Given $Q\subseteq\twome$,  the {\em density of $z$ in $Q$} is given by
\[
\rho(Q\mid z):=
\liminf_n \mu\parb{Q \mid \dbra{z\restr_n}}=
\liminf_n 2^n\cdot \mu\parb{Q\cap \dbra{z\restr_n}}.
\]
A real is a
{\em positive-density point} if it has 
positive-density in every \pz tree that it belongs to.
\end{defi}
\citet{DenjoyBHMN14} showed that a random $z$ is a positive-density point iff $z\not\geq_T\zp$.
Positive-density random reals are not necessarily {\em density-1 reals} (density tends to 1)
and the complete characterization of density-1 reals is an open problem. However 
a  positive-density random real can have arbitrarily high density, in an appropriately chosen \pz tree $P$:
\begin{lem}\label{iKVOwZhlH9}
If $z$ is a positive-density random real and $\epsilon>0$, there exists \pz pruned tree $P$ of finite deficiency such that  $z\in P$
and the $P$-density of $z$ is $>1-\epsilon$.
\end{lem}\begin{proof}
Since $z$ is random, $\exists c\ \forall n\ K(z\restr_n)\geq n-c$. Let $V_0$ be a \ce \pf set such that 
\[
\dbra{V_0}=\dbra{\sqbrad{\sigma}{K(\sigma)< |\sigma|-c}}
\]
and let $P_0$ be the \pz pruned tree consisting of the strings with no prefix in $V_0$. Then $z\in [P_0]$ and
since $z$ is a positive-density real, there exists $\delta>0$ such that the $P_0$-density of $z$ is $>\delta$, so
\[
\forall \tau\prec z: \ \mu_{\tau}(V_0)<1-\delta.
\] 
Let $k$ be such that $(1-\delta)^k<\epsilon$ and let $V:=(V_0)^k$, so for each $\tau\in P_0$, $\mu_{\tau}(V)\leq \parb{\mu_{\tau}(V_0)}^k<\epsilon$.
By Lemma \ref{nBTeELXDu2} (i), $V$ is \ce and contains a prefix of every non-random real.
Let $P$ be the pruned \pz tree such that $[P]=\twome-\dbra{V}$. 

Then $P\supseteq P_0$, $z\in P$,  and $P$ contains only random reals.
So $\mu_{\tau}(P)>1-\parb{1-\mu_{\tau}(P_0)}^k$ and
\[
\tau\prec z\ \Rightarrow\ \mu_{\tau}(P)>1-\parb{1-\mu_{\tau}(P_0)}^k>1-(1-\delta)^k>1-\epsilon
\]
which shows that the density of $z$ in $P$ is $>1-\epsilon$.
\end{proof}
We now show a gap theorem: an incomplete random real can be everywhere branching inside some 
path-incompressible \pz  tree $P$, but the branching density of a Turing-hard random real in such $P$ 
is precisely and considerably more sparse.  
\begin{thm}\label{ZkBuECAo7z}
Given computable increasing $\ell=(\ell_n)$, the following are equivalent:
\begin{enumerate}[\hspace{0.3cm}(a)]
\item $\forall$ path-incompressible \pz  tree $P$, $\exists$  $(\ell, P)$-branching $z\geq_T \emptyset'$
\item $\exists$ \pz path-incompressible tree $P$ and $(\ell, P)$-branching $z\geq_T \emptyset'$
\item $\sum_n 2^{-(\ell_{n+1}-\ell_n)}<\infty$.
\end{enumerate}
If $z\not\geq_T \emptyset'$ is random, $\ell_n=n$, 
$\exists$ path-incompressible \pz  tree $P$ such that $z$ is $(\ell, P)$-branching.
\end{thm}\begin{proof}
By Lemma \ref{2PimWgH6lS} (ii) we get (c)$\to$(a) while (a)$\to$(b) is trivial.
By Lemma \ref{2PimWgH6lS} (i) we get (b)$\to$(c).
For the last clause note that if $z$ has density $>1/2$ in $P$, it is $(\ell, P)$-branching
for $\ell_n=n$. Hence the last clause
follows from the characterization of incomplete reals as positive-density points by 
\citet{DenjoyBHMN14}, and Lemma \ref{iKVOwZhlH9}.
\end{proof}

\subsection{Increasing the density of branching}\label{dVhfnJiXuz}
We are interested in effectively transforming a perfect path-incompressible tree into one with more dense branching, without
significant loss in the deficiency. To this end, we give a positive answer on certain conditions.

Let  $(\ell_n)$ increasing and computable,  and  let
$\TT_{\ell}$ be the pruned trees such that each $\sigma\in T\cap 2^{\ell_n}$ has one or two extensions in 
$T\cap 2^{\ell_{n+1}}$.
The uniform measure $\nu$ on $\TT_{\ell}$ is induced by
\[
\nu([T\restr_{\ell_n}])=\frac{1}{\abs{\TT_{\ell_n}}}
\hspace{0.3cm}\textrm{for $T\in \TT_{\ell}$, where $\TT_{\ell_n}:=\sqbrad{T\restr_{\ell_n}}{T\in\TT_{\ell}}$}
\]
where $T\restr_{\ell_n}:=T\cap 2^{\leq \ell_n}$ and 
$[T\restr_{\ell_n}]$ denotes the set of trees in $\TT_{\ell}$ that  have $T\restr_{\ell_n}$ as a prefix.

Theorem \ref{VlMkthFR3B}
is a special case of the following, for $m_n=n^2$.
\begin{thm}\label{dFNXj6nijs}
Let  $\ell=(\ell_n)$, $m=(m_n)$ be computable and increasing such that
\[
\ell_{n+1}-\ell_{n}\geq m_{n+1}-m_{n}
\hspace{0.3cm}\textrm{and}\hspace{0.3cm}
\sum_n 2^{-(m_{n+1}-m_{n}-n)}<\infty.
\]
%and  $\TT_{\ell}$ be the pruned trees such that each $\sigma\in T\cap 2^{\ell_n}$ has one or two extensions in 
%$T\cap 2^{\ell_{n+1}}$.
There exists a truth-table map $\Phi:\TT_{\ell}\to\TT_{m}$ such that for  $T\in\TT_{\ell}$:
\begin{itemize}
\item if $T$ is path-incompressible, so is $\Phi(T)$
\item with $\nu$-probability 1,  $T$ is $\ell$-perfect and  $\Phi(T)$ is $m$-perfect.
\end{itemize}
\end{thm}
Toward the proof, we  need to be specific regarding the deficiency of the trees, so
 consider a \pz pruned tree containing the $c$-incompressible reals: 
\[
P_c=\{\sigma\ |\ \forall \rho\preceq\sigma, \ K(\rho)\geq |\rho|-c\}
\hspace{0.3cm}\textrm{and}\hspace{0.3cm}
\TT_{\ell}(P_c):=\sqbrad{T\in\TT_{\ell}}{[T]\subseteq [P_c]}.
\]
For Theorem \ref{dFNXj6nijs} it suffices to define a 
truth-table $\Phi:\TT_{\ell}\to\TT_{m}$ such that for  $T\in\TT_{\ell}$:
\begin{enumerate}[\hspace{0.3cm}(a)]
\item if $T$ is path-incompressible, so is $\Phi(T)$: $\exists d\ \forall c\ \Phi(\TT_{\ell}(P_c))\subseteq \TT_m(P_{c+d})$
\item with $\nu$-probability 1,  $T$ is $\ell$-perfect and  $\Phi(T)$ is $m$-perfect
\end{enumerate}
The required map $\Phi$ will be defined by means of a family of sets of strings.
\begin{defi}
Given increasing $m=(m_i), \ell=(\ell_i)$, 
a {\em $(m,\ell)$-family} $\Hs$ is  a family $(H_{\sigma})$ of finite subsets of $\twomel$ indexed by
the $\sigma\in \sqbrad{2^{m_n}}{n\in\Nat}$
such that for each $\sigma \in 2^{m_n}$, $\tau \in 2^{m_{n+1}}$: 
\[
\parb{\sigma\prec\tau\ \Rightarrow\ \dbra{H_{\tau}}\subseteq \dbra{H_{\sigma}}}\ \wedge\ H_{\sigma}\subseteq 2^{\ell_{n}}
\ \wedge\ \mu(H_{\sigma})\leq 2^{-|\sigma|}
\]
Given an $(m,\ell)$-family $H$ define the {\em $(H, m, \ell)$-map}: 
$\Phi(T;n)=\sqbrad{\sigma\in 2^{m_n}}{T\cap H_{\sigma}\neq\emptyset}$.
\end{defi}
If $\sigma\mapsto H_{\sigma}$ is computable, $\Phi$ defines a truth-table map from $\TT_{\ell}$ to $\TT_m$.

It remains to define a 
computable $(m,\ell)$-family $H:=(H_{\sigma})$ such that
conditions (a), (b) above hold for the corresponding truth-table map $\Phi$.

{\bf Construction.}
Let $H_{\lambda}=\{\lambda\}$ and inductively assume that $H_{\sigma}, \sigma\in 2^{m_i}, i<n$ have been defined.
For each $\sigma\in 2^{m_{n-1}}$, $\rho\in 2^{m_n-m_{n-1}}$ let
\[
H_{\sigma\ast \rho}:=\bigcup_{\tau\in H_{\sigma}}\sqbrad{\tau'\in 2^{\ell_n}}{\tau\ast \rho \prec\tau'},
\]
so $\mu(H_{\sigma\ast \rho})=\mu(H_{\sigma})\cdot 2^{-|\rho|}$, $\mu(H_{\sigma})=2^{-|\sigma|}$.
Let $\Phi$ be the $(H, m, \ell)$-map, so
\[
\sigma\in \Phi(T)\iff H_{\sigma}\cap T\neq\emptyset.
\]
Let $\Phi$ be the truth-table functional induced by $(H_{\sigma})$.

{\bf Verification.}
Toward (a), consider a \pf machine $M$ such that
\[
\forall n,c\ \forall \sigma\in 2^{m_n}:\ \parlr{K(\sigma)\leq |\sigma|-c\ \Rightarrow\ \forall \tau\in H_{\sigma}: K_M(\tau)\leq |\tau|-c}.
\]
Such $M$ exists by the Kraft-Chaitin-Levin  theorem (see \cite[\S 3.6]{rodenisbook})
since $\mu(H_{\sigma})=2^{-|\sigma|}$, so the weight of its descriptions  is bounded by 1. 
Let $d$ be such that $K\leq K_M+d$ so
\[
K(\sigma)\leq |\sigma|-c-d\ \Rightarrow\ \forall \tau\in H_{\sigma}: K(\tau)\leq |\tau|-c
\]
for each $n,c$ and $ \sigma\in 2^{m_n}$. Hence for all $\tau$: 
\[
\parB{\tau \in T \cap 2^{\ell_n}\wedga K(\tau) > |\tau| - c}\ 
\Rightarrow\ \forall \sigma \in \Phi(T) \cap 2^{m_n}, \ K(\sigma) > |\sigma| - c - d.
\]
Since $(m_n), (\ell_n)$ are increasing and computable, by Lemma \ref{nBTeELXDu2} (ii)
this proves (a).

For (b), we first show that 
the $\nu$-probability of $T$ containing an isolated path is 0.
The probability that $\sigma\in T\cap 2^{\ell_n}$ does not branch at the next level 
is the probability that two independent trials with replacement pick the same extension, which is
$2^{-(\ell_{n+1}-\ell_n)}$.
Since $|T\cap 2^{\ell_n}|\leq 2^{n}$, the probability that this occurs for some $\tau\in T\cap 2^{\ell_n}$
is $\leq 2^{n-(\ell_{n+1}-\ell_n)}$. 
By the hypothesis   
\[
\sum_n 2^{-(\ell_{n+1}-\ell_n-n)}<\infty
\]
so by the first Borel-Cantelli lemma, with probability 1 the tree $T$ is $\ell$-perfect.

For (b), it remains to show that the $\nu$-probability of $\Phi(T)$ containing an isolated path is 0.
If $\sigma\in \Phi(T)\cap 2^{m_n}$ does not branch at the next level, 
the corresponding $\tau\in H_{\sigma}\cap T$ gets both of its two extensions from the same
$H_{\sigma\ast \rho}, \rho\in 2^{m_{n+1}-m_n}$.
The probability of this event $E_{\tau}$ is $2^{-(m_{n+1}-m_n)}$ and,
since $|T\cap 2^{\ell_n}|\leq 2^{n}$, the probability that $E_{\tau}$ occurs for some $\tau\in T\cap 2^{\ell_n}$
is $\leq 2^{n-(m_{n+1}-m_n)}$. 
By the hypothesis there exists $b$ such that   
\[
\sum_n 2^{-(m_{n+1}-m_n-n)}<b
\]
so the probability that the above event occurs in more than $2^c$ many levels of $T$ is $\leq b\cdot 2^{-c}$.
By the first Borel-Cantelli lemma, with probability 1, $\Phi(T)$ is $m$-perfect.
This completes the verification of (b) and the proof of Theorem \ref{dFNXj6nijs}.

\section{Conclusion and discussion}\label{uPNq4hrByE}
We studied the extent to which the branching in a path-incompressible tree can be effectively altered, without
significant deficiency increase.
We showed that there is a path-incompressible proper tree that does not compute any path-incompressible prefect tree with
a computable upper bound on the oracle-use.
 
We also explored the limits of effective densification of perfect path-incompressible trees, 
and in this context the following  question seems appropriate:
given computable increasing  $\ell=(\ell_n)$, $m=(m_n)$ with $\ell_n\gg m_n\gg n^2$,
is there an $\ell$-perfect path-incompressible tree which does not compute 
any $m$-perfect path-incompressible tree?

Our methodology relied on the use of  hitting-families of open sets, for expressing maps from trees to trees.
We suggest that this framework 
can give analogous separations between classes of trees of different Cantor-Bendixson rank.
Applications  are likely in the study of compactness in fragments of second-order arithmetic 
\citep{deniscarlsch21, treeout, ChongLi2019, pamiller}.  

%Is there a $\nu$-measure on $\TT$ such that for each sufficiently $\nu$-random $T\in\TT$:
%\begin{itemize}
%\item $T$ is path-incompressible
%\item $T$ does not compute any perfect pathwise incompressible tree.
%\end{itemize}
%Note that $\nu$ cannot be computable (Axon, Poisson etc.)

%\bibliographystyle{abbrvnat}
%\bibliography{dimtree}
\end{document}